\documentclass{proc-l}
\usepackage{amssymb,amsmath}
\usepackage[dvips]{graphics}
\usepackage[dvips]{color}
\usepackage{graphicx}
\usepackage{hyperref}
\usepackage[english]{babel}
\usepackage[T1]{fontenc}       		
\usepackage{mathrsfs}   
\usepackage{comment}
\usepackage[textsize=tiny]{todonotes}
\usepackage{xcolor}

\newtheorem{Theorem}{Theorem}[section]
\newtheorem{Lemma}[Theorem]{Lemma}
\newtheorem{Proposition}[Theorem]{Proposition}
\newtheorem{Corollary}[Theorem]{Corollary}
\newtheorem{Conjecture}[Theorem]{Conjecture}

\theoremstyle{definition}
\newtheorem{Definition}[Theorem]{Definition}
\newtheorem{Example}[Theorem]{Example}
\newtheorem{Remark}[Theorem]{Remark}
\newtheorem{Claim}{Claim}

\newtheorem{innercustomgeneric}{\customgenericname}
\providecommand{\customgenericname}{}
\newcommand{\newcustomtheorem}[2]{%
	\newenvironment{#1}[1]
	{%
		\renewcommand\customgenericname{#2}%
		\renewcommand\theinnercustomgeneric{##1}%
		\innercustomgeneric
	}
	{\endinnercustomgeneric}
}

\newcustomtheorem{customthm}{Theorem}

\newcommand\n{\mathbb N}

\usepackage[mathcal]{eucal}

\newcommand{\N}{{\mathbb N}}
\newcommand{\Z}{{\mathbb Z}}

\begin{document}

\title[Linear Dynamics Induced by  Odometers]{Linear Dynamics Induced by Odometers}

\author{D. Bongiorno}
\address{Dipartimento di Ingegneria, Universit\`{a} degli Studi di Palermo}
\curraddr{Viale delle Scienze, 90100 Palermo, Italy}
\email{donatella.bongiorno@unipa.it}

\author{E. D'Aniello}
\address{Dipartimento di Matematica e Fisica, Universit\`{a} degli Studi della Campania "Luigi Vanvitelli"}
\curraddr{Viale Lincoln 5, 81100 Caserta, Italy}
\email{emma.daniello@unicampania.it}

\author{U. B. Darji}
\address{Department of Mathematics, University of Louisville}
\curraddr{Louisville, KY 40292, USA}
\email{udayan.darji@louisville.edu}

\author{L. Di Piazza}
\address{Dipartimento di Matematica ed Informatica, Universit\`{a} degli Studi di Palermo}
\curraddr{Via Archirafi 34, 90100 Palermo, Italy}
\email{luisa.dipiazza@unipa.it}



\thanks{This research has been partially supported by the Gruppo Nazionale per l'Analisi Matematica, la Probabilit\`a e le loro Applicazioni (GNAMPA) of the Istituto Nazionale di 
Alta Matematica (INdAM) (Project 2018   ``Metodi di approssimazione mediante somme integrali e sistemi dinamici caotici'').} 
\date{}
\maketitle
     
     \begin{abstract}
       Weighted shifts are an important concrete class of operators in linear dynamics. In particular, they are an essential tool in distinguishing a variety dynamical properties. Recently, a systematic study of dynamical properties of composition operators on $L^p$ spaces has been initiated. This class of operators includes weighted shifts and also allows flexibility in construction of other concrete examples. In this article, we study one such concrete class of operators, namely composition operators induced by measures on odometers. In particular, we study measures on odometers which induce mixing and transitive linear operators on $L^p$ spaces.
     \end{abstract}
\maketitle
\section{Introduction}
Linear dynamics is a relatively recent area of mathematics which lies at the intersection of operator theory and dynamical systems. A flurry of intriguing results have been obtained in this area starting with investigation of transitivity and mixing. Recent investigations include concepts such as Li-Yorke, Devaney and distributional chaos, invariant measures, ergodicity and frequently hypercyclic, expansive, hyperbolic, shadowing and structural stability. We refer the reader to books \cite{BM} and \cite{GP} for general information on the topic. 

 A class of operators which plays a key  role in linear dynamics is the class of weighted shifts. This class was introduced by Rolewicz \cite{RO}. Let us briefly recall their definitions. Let $A$ be $\N$ or $\Z$. Let $\{w _i\}_{i \in A}$ be  a bounded sequence of positive reals called the
{\em weight sequence}. Then, the backward weighted shift on $X = \ell_p(A)$ $(1 \leq p < \infty)$ or $X = c_0(\Z)$ is a mapping $B_w: X \rightarrow X$ defined by
\[
B_w\big((x_n))(i) = w_{i+1}x_{i+1}.
\]
Salas (\cite{S1}, \cite{S2}) initiated the study of weighted shifts which are transitive and mixing. His results were generalized and extended by various authors \cite{G0}, \cite{MP}, \cite{CS}. Characterizations obtained in these articles readily allow examples and counterexamples. For example, for weighted shifts on $\N$,  the operator $B_w$ is transitive if and only if 
\[ \sup_n w_1 \ldots w_n = \infty, \]
and, mixing if and only if
\[ \lim_n w_1 \ldots w_n = \infty. \]
By these characterizations one can construct with ease an operator which is transitive but not mixing. Characterizations for weighted shifts which are Li-Yorke chaotic were given in \cite{BBM} and  \cite{BDP}. In \cite{BCDMP}, using their characterization of weighted shifts which are expansive, the authors were able to settle negatively an open problem, i.e, there exists an operator with the  shadowing property which is not hyperbolic. In \cite{BerMess}, using their characterization of shadowing, the authors were able to construct a structurally stable operator with the shadowing property which is not hyperbolic.  
In \cite{BR}, the authors characterize weighted shifts on $c_0$ which are frequently hypercyclic, settling negatively the open problem of whether ${\mathcal U}$-frequently hypercyclic and frequently hypercyclic are equivalent notions. These characterizations were used \cite{gm} to construct a frequently hypercyclic weighted shift on $c_0$ which admits no invariant ergodic measure with full support.

Authors of \cite{BDP} and \cite{BDP2} initiated a  systematic study of dynamical properties of composition operators. In \cite{BDP}, necessary and sufficient conditions were given for a composition operator to be transitive and mixing. Necessary and sufficient conditions for a composition operator to be Li-Yorke chaotic were given in \cite{BDP2}. These results include earlier results of Salas as every weighted shift is topologically conjugate to a composition operator. The class of composition operators also include Rolewicz type operators introduced in \cite{BDD}. 

In this article, we study a special class of composition of operators. We study composition operators induced by odometers.

Odometers appear by various names in a wide range of topics. They are also called adding machines or solenoids. One of the earliest uses of odometers in ergodic theory was by Ornstein \cite{ORN}. He showed that there is an invertible transformation with a non-singular probability measure $\mu$ for which there is no $\sigma$-finite invariant measure $\nu$ which is equivalent to $\mu$. (See (\cite{DASI}: Example 5) for Ornstein example with the full proof). 

Example 2.4 in \cite{DS} contains a construction in modern terminology. Odometers also appear abundantly in topological dynamics. For example, it was shown in \cite{HOC} that a generic transitive homeomorphism of the Cantor space is topologically conjugate to the universal odometer. Also in \cite{DDS} and \cite{DHS} it was shown that the omega-limit set generated by a generic map and a generic point on a manifold is topologically conjugate to the universal odometer. 

We choose an arbitrary odometer. On each coordinate space, we put a probability measure $\mu_i$ and consider the product measure $\mu$ on the odometer. We consider  the composition operator $T_f: L^p (\mu) \rightarrow L^p (\mu)$ defined by $T_f(\varphi) = \varphi \circ f$ where $f$ is the $+1$ map on the  odometer.  We give conditions on the $\mu_i$'s which guarantee that the composition operator is transitive and mixing. The spirit of our approach is that of weighted shifts, i.e., a concrete class of operators which hopefully becomes an indispensable tool in linear dynamics. 

The article is organized as follows: in Section~2 we recall basic definitions and background results; in Section~3, we state our main results; in Sections 4 and 5 we give their proofs. 
\section{Basic Notions and Background results}
Throughout the paper by $\N$ we denote the non-negative integers. 
We start by recalling notions of transitivity and  mixing. 
 \subsection{Transitivity and Mixing} 
 \begin{Definition} (Topological Transitivity) A bounded linear operator $T: {\mathcal X}  \rightarrow {\mathcal X}$ acting on a Banach space ${\mathcal X}$ is \textit{topologically transitive} if for 
 any pair $U, V \subseteq {\mathcal X}$ of nonempty open sets, there exists an integer $k \geq 1$ such that $T^{k}(U) \cap V \not= \emptyset$. 
\end{Definition}
In a separable Banach space, a bounded linear operator is transitive if and only if it is hypercyclic, i.e., the operator has a dense orbit. For more information on linear dynamics we refer the reader to \cite{BM} and \cite{GP}.

 \begin{Definition} (Topological Mixing) A bounded linear operator $T: {\mathcal X}  \rightarrow {\mathcal X}$ acting on a Banach space ${\mathcal X}$ is \textit{topologically mixing} if 
 for any pair $U, V \subseteq {\mathcal X}$ of nonempty open sets, there exists an integer $k_{0} \geq 1$ such that $T^{k}(U) \cap V \not= \emptyset$ for every $k \geq k_{0}$. 
 \end{Definition}

We often drop the adjective topological and simply say transitive or mixing. 
It is clear that mixing implies transitivity. An abundance of examples exist showing that the converse is false. For instance, see  \cite{SM} and \cite{BM2}. 

\subsection{Composition Operators}
Let $(X, {\mathcal S}, \mu)$ be a $\sigma$-finite measure space and $g: X \rightarrow X$ be measurable.
Let $1 \leq p < \infty$. Then, $T_{g}: L^{p}(\mu) \rightarrow L^{p}(\mu)$ defined by $T_{g}: \phi \mapsto {\phi} \circ g$ is a bounded linear operator if and only if it satisfies the condition:
\[\tag{*} \exists c >0 \ \ \ \left [B \in {\mathcal S} \implies \mu(B) \geq c \mu(g^{-1}(B) ) \right].  \]
For a proof of this and general information about composition operators, we refer the reader to \cite{SM}. 

Characterizations of topological transitivity  and topological mixing for composition operators were given in \cite{BDP}. We will use these characterizations extensively. Below we state them in the specific form we will use.

\begin{Theorem}\label{MixingBDP} Let $X=(X,{\mathcal S},\mu)$ be a finite measure space and $g:X\to X$ be an one-to-one bimeasurable transformation  satisfying $(*)$. 
Then, the composition operator $T_g: L^p(X) \rightarrow L^p(X)$ is topologically mixing if and only if 
for each $\varepsilon>0$, there exist $k_0\ge 1$ and measurable sets $\{B_k\}_{k=k_0}^\infty$ such that 
\[\mu(X{\setminus} B_k)< \varepsilon\quad\textrm{and}\quad B_k\cap g^k(B_k)=\emptyset\quad\textrm{for every}\quad k\ge k_0;\]
\end{Theorem}

\begin{Theorem}\label{TransBDP}  Let $X=(X, {\mathcal S},\mu)$ be a finite measure space and $g:X\to X$ be an one-to-one bimeasurable transformation  satisfying $(*)$. Then, the composition operator 
$T_g: L^p(X) \rightarrow L^p(X)$ is topologically transitive if and only if 
for each $\varepsilon>0$, there exist $k\ge 1$ and a measurable set $B\subseteq X$ with 
\[\mu(X{\setminus} B)< \varepsilon\quad\textrm{and}\quad B\cap g^k(B)=\emptyset.\]
\end{Theorem}
\subsection{Odometers} Let ${\bf \alpha} = ({\alpha}_{0},  {\alpha}_{1}, \ldots)$  be a sequence of integers with 
${\alpha}_{i} \geq 2$. Define  
\[{\mathcal A}_{\alpha} : = {\prod}_{i=0}^{\infty} {A}_{i},\]
 where ${A}_{i} : = \{0, 1, \ldots, {\alpha}_{i} -1\}$. We consider ${\mathcal A}_{\alpha}$ endowed with the product topology. Hence, the obtained topological space is homeomorphic to the Cantor space.  
 We recall that  the open subsets of ${\mathcal A}_{\alpha}$ are countable unions of disjoint basic cylinders, i.e., sets of the form  
\[[a_{0}, \dots, a_{i}] := \{(x_{0}, \dots, x_{i}, x_{i+1}, \dots) \in {\mathcal A}_{\alpha}: x_{0} = a_{0}, \dots, x_{i} = a_{i}\}.\]
 
 We consider the following addition with ``carry over"  to the right: if $x = (x_{0}, x_{1}, \ldots)$ and $y=(y_{0}, y_{1}, \ldots)$ are elements of ${\mathcal A}_{\alpha}$, then
 \[x+y = z = (z_{0}, z_{1}, \ldots),\]
where $z_{i} = ( x_{i} + y_{i} + {\varepsilon}_{i}) \text{mod}({\alpha}_{i})$ where $\varepsilon_{0}=0$ and for $i \ge 0$
 \[ {\varepsilon}_{i+1} = \left\{ \begin{array}{ll}
0   & \mbox{if } x_{i} +y_{i} + {\varepsilon}_{i} < {\alpha}_{i} \\
1  & \mbox{otherwise.} \\
\end{array}
\right.\] 
$\varepsilon_i$ is called ``carry over'' in the $i^{th}$ position.

We denote by $f_{\alpha}$ the map ``+1'', that is 
$f_{\alpha}(x_{0}, x_{1}, \ldots) =({x}_{0}, x_{1}, \ldots)+(1, 0, 0, \ldots).$
The pair $({\mathcal A}_{\alpha}, f_{\alpha})$ is a dynamical system known in various contexts as a
solenoid,  adding machine or odometer \cite{BK} and \cite{BK2}. We refer to $f_{\alpha}$ as an {\it odometer}. We point out that, as $f_\alpha$ preserves distance,  it is one to one. Morever, the orbit of $(0,0,\ldots)$ under $f_{\alpha}$ is dense in ${\mathcal A}_{\alpha}$. Hence, $ f_{\alpha}$ is  a bijection. 

If  $k\geq 0$ is an integer,  we identify it with  its  representation $(k_{0}, k_{1}, \dots)$ in terms of $\alpha_i$'s.  More specifically, 
$k = k_0+ \sum_{i=1}^{\infty} k_i \beta _i$ where
$\beta _i = \prod _{j=0}^{i-1} \alpha _j$. This representation is unique. We use $k$ and its representation interchangeably, without explicitly stating so. Moreover, if $x=(x_{0}, x_{1}, \dots) \in {\mathcal A}_{\alpha}$, then   by definition $f^{k}(x)= (x_{0}, x_{1}, \dots) + \underbrace{(1, 0, \dots) + \cdots + (1, 0, \dots)\,}_\text{k times}$. Therefore, $f^{k}(x)= (x_{0}, x_{1}, \dots) + (k_{0}, k_{1}, \dots)$, where the addition is made with the ``carry over'' to the right as described earlier. 

\subsection{Measures on Odometers} \label{subsec}  Let $({\mathcal A}_{\alpha}, f_{\alpha})$ be an odometer.  By ${\mathcal B}({\mathcal A}_{\alpha})$ we denote the $\sigma$-algebra of all Borel subsets of ${\mathcal A}_{\alpha}$. For each $i$, let ${\mu}_{i}$  be a probability measure on $A_{i}$ with ${\mu}_{i}(a) >0$ for all $a \in A_{i}$. Consider  the infinite product probability space 
$({\mathcal A}_{\alpha}, {\mathcal B}({\mathcal A}_{\alpha}), \mu): = \otimes_{i=0}^{\infty} (A_{i}, {\mu}_{i})$. 
We will study $T_{f_{\alpha}}: L^p(\mu) \rightarrow L^p(\mu)$. In particular, we will give necessary and sufficient conditions on $\{\mu_i\}$ which guarantee topological transitivity and topological mixing of $T_{f_{\alpha}}$. We would like to point out that $f_{\alpha}$, in general, does not preserve $\mu$.

Measures of these types on odometers are well-studied. For example, see the survey article \cite{DASI} on ergodic theory of non-singular transformation. 
If ${\prod}_{i=0}^{\infty} \max \{{\mu}_{i}(a): a \in A_{i}\}=0$, then $\mu$ turns out to be non-atomic, $f_{\alpha}$ is non-singular with respect to $\mu$, and $f_{\alpha}$ is ergodic.
In such situations, $f_{\alpha}$ is called the {\it non-singular odometer} associated with ${({\alpha}_{i}, {\mu}_{i})}_{i=0}^{\infty}$. 

The importance of non-singular odometers is also due to the fact that each non-singular odometer is a Markov odometer \cite{DASI}, and, as  it is well-known,      
every ergodic non-singular transformation is orbit-equivalent to a Markov odometer \cite{DH}.


\section{Main Results}
Throughout the paper, we work with only one odometer $\alpha = ({\alpha}_{0},  {\alpha}_{1}, \ldots)$ at a time and hence we call the map $f$ instead of $f_{\alpha}$ and use ${\mathcal A}$ instead of ${\mathcal A_{\alpha}}$.
For each $i$, we fix ${\mu}_{i}$ a probability measure on $A_{i}$ with support $A_i$ and we call $\mu$ the product of $\mu_i$'s. To ensure that $T_{f}$ is continuous, we must 
guarantee that $f$ satisfies Condition $(*)$. To this aim,  for $i \in \N $ and $j \in A_i$, define
 \[{\lambda}_{i}(j) := \frac{{\mu}_{i}(j)}{{\mu}_{i}(j-1)},\]
meaning that $j-1 = {\alpha}_{i} -1$ if $j=0$. Then, $\{\mu_i\}$ satisfies Condition $(*)$ if  and only if the following holds (see \cite{DASI}):
\[\tag{$\diamondsuit$}  0 < \inf  \left \{{\lambda}_{l}(j) {\prod}_{i=0}^{l-1} {\lambda}_{i}(0) : l \geq 1, j \in  A_l \right \}.\]
It is clear that, for any given sequence of $\{\alpha_i\}$, if we choose $\mu_i$ to be uniform distribution on $A_i$, then $\{\mu_i\}$ satisfies Condition ($\diamondsuit$). In general, it may not be obvious that $\{\mu_i\}$ satisfies Condition ($\diamondsuit$) and it needs to be verified.

In order to state the main results of the paper, we introduce some notation. 
For $i \in \n$, let 
\begin{align*}
{\eta}_{i} & : =   \max \{{\mu}_{i}(j): j \in A_i\},\\
{\delta}_{i}& : =  \min \{{\mu}_{i}(j): j \in A_i\}.
\end{align*}
For $0 \leq i \leq j$,  we denote by $p_{i,j}: {\mathcal A} \rightarrow {\prod}_{s=i}^j A_{s}$ the ``natural'' projection, 
that is $p_{i,j}((x_{0}, x_{1}, \dots, x_{i}, ..., x_{j}, x_{j+1}, \dots)) = (x_{i}, ..., x_{j})$, and by ${\mu}_{{i, j}}$ the product measure  
$\prod_{s=i}^j {\mu}_{s}$ on ${\mathcal A}_{i,j} : = {\prod}_{s=i}^j A_{s}$. We define 
\[\psi_{i,j}:=  \max  \left \{\mu_{{i, j}}(A): A \subseteq {\mathcal A}_{i,j} \text{ such that } \  \exists  k \in  {\mathcal A}_{i,j} \text{ with } (A+k) \cap A = \emptyset \right \},\]
with $A+k:= \{(a+k) \bmod {\alpha}_{i} : a \in A\}$, where the addition is the ``carry over'' to the right defined before. (If we have a carry over at the last step, we ignore it.) Moreover, if we write, for any integer $i$, $\psi_{i}$ instead of ${\psi}_{i,i}$, we clearly have that
\[\psi_i:=  \max \{\mu_i(A): A \subseteq A_i \text{ such that } \exists \  0 < k <   \alpha_i  \text{ with } (A+k) \cap A = \emptyset\},\]
where $A+k:= \{ (a+k) \bmod {\alpha}_{i} : a \in A\}$.  It is immediate that $0 <  {\eta}_{i} \leq \psi_i.$ Finally,
we define
\[\rho _n := \prod_{i=0}^n \left (\eta_i/\delta_i \right ).\]

Next, we state the main results of the paper. We always assume that $\mu$ satisfies Condition ($\diamondsuit$), unless otherwise stated. This is equivalent to saying that $T_f$ is well-defined and continuous.  

In the mixing case, we have an explicit characterization when the sequence $\{\alpha_i\}$ is bounded and $T_f$ is continuous.

\begin{Theorem} \label{MixingMainBounded} {\em (Mixing  Characterization)} Suppose $\{\alpha_i\}$ is bounded and $\{\mu_i\}$ is such that  $T_f$ is continuous. Then, the composition operator  $T_{f}$ is topologically mixing if and only if $\lim_{i}{\eta}_{i}  = 1.$
\end{Theorem}
   
One may wonder if there exists any such $\{\alpha_i\}$  and $\{\mu_i\}$ as in the statement of Theorem~\ref{MixingMainBounded}. The following theorem shows that such is the case if $\{\alpha_i\}$ does not grow too fast. 
\begin{Theorem} {\em (Mixing Existence)} \label{MixingExistence}
Let $\{\alpha_i \}$ be such that
\[ \sup_n \frac{\alpha_{n+1}}{ \prod_{i=0}^{n} \alpha_i } < \infty. \]
Then, there exists $\{\mu_i\}$ satisfying ($\diamondsuit$) such that $T_f$ is topologically mixing. In particular, such is the case when $\{\alpha_i\}$ is bounded. 
\end{Theorem}
One may also wonder if Theorem~\ref{MixingMainBounded} holds if $\{\alpha_i\}$ is unbounded. The following example, with an intricate construction, shows that Theorem~\ref{MixingMainBounded} can fail spectacularly if $\{\alpha_i\}$ is unbounded.
\begin{Example}\label{MixingExample}
There exist $\{\alpha_i\}$ and $\{\mu_i\}$ such that $T_f$ is mixing and $\lim_i \eta_i =0$. 
\end{Example}
In the transitive case, we have the following characterization and its corollary which yields a simple sufficient condition. 
\begin{Theorem} \label{TRANSCHAR}  {\em (Transitivity Characterization)} 
Suppose $\{\alpha_i\}$ and $\{\mu_i\}$ are such that  $T_f$ is continuous. Then, the composition operator  $T_{f}$ is topologically transitive if and only if $\limsup_{i\leq j} \{{\psi}_{i,j}\} =1$. 
\end{Theorem}
\begin{Corollary} \label{CorTRANSSUF} {\em (Transitivity Sufficient  Condition)}
Suppose $\{\alpha_i\}$ and $\{\mu_i\}$ are such that  $T_f$ is continuous. 
If $\limsup_{i} {\eta}_{i} =1$, then $T_{f}$ is topologically transitive.
\end{Corollary}
An example given in \cite{BDP}, Section 3.3, shows that in the setting of odometers there are topologically transitive  operators which are not topologically mixing.  The following theorem shows that this can be done in a very general setting. 
\begin{Theorem}\label{TransExistence1} {\em (Transitive Existence 1)} Suppose $\{\alpha_n\}$ is such that $\liminf_n \alpha_n < \infty$. Then, there exists $\{\mu_i\}$ such that $T_f$ is continuous,  topologically transitive and not mixing. 
\end{Theorem}
One may wonder what happens in the case  $\lim_n \alpha_n = \infty$ as far as the existence of transitive but not mixing operators is concerned. The following theorem shows that one can find such operators when $\{\alpha_n\}$ does not grow too fast.
\begin{Theorem}\label{TransExistence2} {\em (Transitive Existence 2)}
Suppose $\{\alpha_n\}$ is such 
\[  \sup_n \frac{\alpha_{n+1}}{ \prod_{i=0}^{n} \alpha_i } < \infty. \] Then, there exists $\{\mu_i\}$ such that $T_f$ is continuous, topologically transitive and not mixing. 

\end{Theorem}

The following condition is necessary for $T_f$ to be topologically transitive. 
\begin{Theorem} \label{TRANSNEC}  {\em (Transitivity Necessary Condition)} 
Suppose $\{\alpha_i\}$ and $\{\mu_i\}$ are such that  $T_f$ is continuous. If  $T_f$ is topologically transitive, then 
$\limsup_{n} {\rho}_n = \infty.$
\end{Theorem}

For the sake of transparency and concreteness, we  now state Theorem \ref{MixingMainBounded} and Theorem \ref{TransExistence2} in the case where $\alpha_{i} = 2$ for each $i$. In this case, 
${\mathcal A}_{\alpha}$ is the dyadic odometer, topologically a Cantor set. Moreover, for $i \geq 0$, 
\[{\eta}_{i}   =   \max \{{\mu}_{i}(0), {\mu}_{i}(1)\} =  \max \{{\mu}_{i}(0), 1 - {\mu}_{i}(0)\}\]
and, for $n \geq 0$,   
\[\frac{\alpha_{n+1}}{ \prod_{k=0}^{n} \alpha_k } = \frac{2}{2^{n+1}} = \frac{1}{2^{n}}.\] 
Hence, Theorem \ref{MixingMainBounded} and Theorem \ref{TransExistence2} simplify to

\begin{customthm}{3.1'} Suppose $\alpha_i = 2$ for each $i$ and $\{\mu_i\}$ is such that  $T_f$ is continuous. Then, the composition operator $T_{f}$ is topologically mixing if and only 
if $\lim_{i} \max \{{\mu}_{i}(0), 1- {\mu}_{i}(0)\}  = 1.$
\end{customthm}

\begin{customthm}{3.7'}
Suppose $\alpha_i=2$  for each $i$. Then, there exists $\{\mu_i\}$ such that $T_f$ is continuous,  topologically transitive and not mixing.
\end{customthm}

\section{Proof of Mixing Results}
In order to proceed with the proof of Theorem~\ref{MixingMainBounded}, we introduce the following notation: \\

Let $W \subseteq {\mathcal A}$, $k, n \in \N$, $k \geq 1$, and $t \in \{0,1\}$.
Define \[C(k, W, n, t) : = \{x  \in W: 
k+x \text{ has carry over $t$ in the $n^{th}$ position}\},\]
and observe that 
\[W = C(k, W, n, 0) \cup C(k, W, n, 1), \ \ \  \forall n \in {\N}.\]

Below is the proof of the ``if'' part of Theorem~\ref{MixingMainBounded}.
\begin{Lemma} \label{MixingForward} Suppose $\{\alpha_i\}$ and $\{\mu_i\}$ are such that  $T_f$ is continuous.  If $\lim_{i}{\eta}_{i}  = 1$, 
then $T_{f}$ is topologically mixing.  
\end{Lemma}
\begin{proof} 
We apply Theorem \ref{MixingBDP}. Let $\varepsilon >0$. Let $l$ be such that if $i \geq l$, then ${\eta}_{i}   {\eta}_{{i+1}} > 1- \varepsilon$. 
Let $k_{0} =  {\prod}_{i=0}^{l} {\alpha}_{i}$. Let $k > k_{0}$. Then, the representation of $k$ in terms of $\alpha_i$'s has at least  one non-zero 
element in the ${(l+1)}^{th}$ position or beyond. 
Let $j$ be the largest such index. In particular, $j \geq l$. 
Define 
\[B_{k} = \{(x_{0}, x_{1}, \dots) \in {\mathcal A}: x_{j} = a_{j}, x_{j+1} =a_{j+1}\}, \]
where $\mu_{j}(a_{j}) = {\eta}_{j}$ and $\mu_{j+1}(a_{j+1}) = {\eta}_{j+1}$. 
Then,
$\mu(B_{k}) ={\eta}_{j}   {\eta}_{{j+1}} > 1- \varepsilon$,
implying that \[\mu({\mathcal A} \setminus B_{k}) < \varepsilon.\]  Moreover, we have 
\begin{align*}
f^{k}(B_{k}) & =  f^k( C(k, B_{k}, j, 0)) \cup f^k(C(k, B_{k}, j, 1)) \\
& \subseteq  \{x \in {\mathcal A}: x_{j} \not= a_{j} \} \cup \{x \in {\mathcal A}: x_{j+1} \not=a_{j+1} \}.\\
\end{align*}
Indeed, let $y= (y_{0}, y_{1}, \dots) \in B_{k}$. If $y \in  C(k, B_{k}, j, 0)$, then $f^k(y) \in  \{x \in {\mathcal A}: x_{j} \not= a_{j} \} $. 
If $y \in  C(k, B_{k}, j, 1)$ and $f^k(y) \notin \{x \in {\mathcal A}: x_{j} \not= a_{j} \}$, then $f^k(y)$ has $1 + a_{j+1}$ in the ${(j+1)}^{th}$ position and therefore 
$f^k(y) \in  \{x \in {\mathcal A}: x_{j+1} \not=a_{j+1} \}$. \\
Hence, as
\[B_{k} \cap \left (\{x \in {\mathcal A}: x_{j} \not= a_{j} \} \cup \{x \in {\mathcal A}: x_{j+1} \not=a_{j+1} \} \right)= \emptyset,\]
we have that 
 \[f^{k}(B_{k}) \cap B_{k} = \emptyset.\]
Therefore, we have proved that, in correspondence with $\epsilon$, there exist $k_{0} \ge 1$ and measurable sets $\{B_k\}_{k> k_{0}}^\infty$ such that 
\[\mu(X{\setminus} B_k)< \varepsilon\quad\textrm{and}\quad B_k\cap f^k(B_k)=\emptyset\quad\textrm{for every}\quad k> k_{0}.\]
By Theorem \ref{MixingBDP}, $T_{f}$ is mixing. 
\end{proof}

We will use the following three lemmas in order to prove the ``only if'' part of Theorem \ref{MixingMainBounded}.
\begin{Lemma}\label{2ptsNotMixing} Suppose $\{\alpha_i\}$ and $\{\mu_i\}$ are such that  $T_f$ is continuous. 
Let $l \in {\mathbb N}$. Let
$a, b \in A_l$ be two distinct elements of $A_l$. Let $0 < \varepsilon < \frac{1}{16} \min \{\mu_l(a),\mu_l(b)\}$. Then, there exists $m \ge  {\prod}_{i =0}^{l-1} {\alpha}_{i}$ such 
 that, for all $B \subseteq {\mathcal A}$  with $\mu(B) > 1 - \varepsilon$,
we have that $f^{m}(B) \cap B \not= \emptyset$.
\end{Lemma}
\begin{proof}
Assume the hypotheses. Let $m \in \N$ be such that the representation of $m$ has $|a -b|$ in the $l^{th}$ position and zero elsewhere. Note that $m \ge  {\prod}_{i =0}^{l-1} {\alpha}_{i}$.

 Let $B$ be such that $\mu(B) > 1 -\varepsilon$. For each $t\in A_l,$ define $B_t$ as those elements of $B$ which
have $t$ in the $l^{th}$ position. We now claim that $\mu(B_j) \ge \frac{15}{16} \mu_l(j)$, $j \in \{a,b\}$. 
If not, we have that 
\begin{align*}
\mu (B) &= \mu(B_{j}) + \sum_{t \neq j} \mu (B_t) \\
& < \frac{15}{16} \mu_l(j) + \sum_{t \neq j} \mu_l(t)\\
& = 1 -\frac{1}{16} \mu_l(j)\\
& < 1 - \varepsilon,
\end{align*}
contradicting our hypothesis on $B$. 

Let $p: A_{\alpha} \rightarrow {\prod}_{i \in {\mathbb N} \setminus \{l\}} A_{i}$ be the ``natural'' projection, that is $p((x_{0}, x_{1}, \dots, x_{l-1}, x_{l}, x_{l+1}, \dots)) = 
(x_{0}, x_{1}, \dots, x_{l-1}, x_{l+1}, \dots)$, and ${\mu}_{{\mathbb N} \setminus \{l\}}$ be the product measure  $\prod_{i \in {\mathbb N} \setminus \{l\}} {\mu}_{i}$ on ${\prod}_{i \in {\mathbb N} \setminus \{l\}} A_{i}$. Define 
\[C_{a} = p(B_{a}) \text{   and    }C_{b} = p(B_{b}).\]
Hence 
\begin{align*}
\mu(B_{a})  & =  {\mu}_{l}(a) {\mu}_{{\mathbb N} \setminus \{l\}}(C_{a})  >   \frac{15}{16}{\mu}_{l}(a), 
\end{align*}
and 
\begin{align*}
\mu(B_{b})  & =  {\mu}_{l}(b) {\mu}_{{\mathbb N} \setminus \{l\}}(C_{b}) >  \frac{15}{16}{\mu}_{l}(b).
\end{align*}
In particular, 
 \[{\mu}_{{\mathbb N} \setminus \{l\}}(C_{a})  >  \frac{15}{16}  \ \ \ \ \ {\mu}_{{\mathbb N} \setminus \{l\}}(C_{b})  >  \frac{15}{16},\]
implying that  
\[C_{a} \cap C_{b} \not= \emptyset.\]
Let $(x_{0}, x_{1}, \ldots, x_{l-1}, x_{l+1}, \dots) \in C_{a} \cap C_{b}$. 
Then,
\begin{align*}
(x_{0}, x_{1}, \ldots, x_{l-1}, a, x_{l+1}, \dots)  & \in B_{a} \\ (x_{0}, x_{1}, \ldots, x_{l-1}, b, x_{l+1}, \dots)   & \in B_{b}.   
\end{align*}

If $b < a$, then $f^{m}(B_{b})\cap B_{a} \not= \emptyset$ as $(x_{0}, x_{1}, \ldots, x_{l-1}, a, x_{l+1}, \dots) \in f^{m}(B_{b})\cap B_{a}$. Similarly, if $a < b$, then $f^{m}(B_{a})\cap B_{b} \neq \emptyset$. In either case, we have that $f^{m}(B) \cap B \not= \emptyset$, completing the proof. 
\end{proof}
 \begin{Lemma} \label{LemOverlap} Suppose $\{\alpha_i\}$ and $\{\mu_i\}$ are such that  $T_f$ is continuous.  Let $l \in {\mathbb N}$. Let $0 < \varepsilon < \frac{1}{16} \frac{1 - {\eta}_{l}}{ \alpha_l}$. Then, there exists $m \ge  {\prod}_{i =0}^{l-1} {\alpha}_{i}$ such 
 that for all $B \subseteq {\mathcal A}$  with $\mu(B) > 1 - \varepsilon$,
we have that $f^{m}(B) \cap B \not= \emptyset$.
 \end{Lemma}  
 \begin{proof}
 
Let $a \in A_{l}$ be such that ${\mu}_{l}(a) = {\eta}_{l}$. Let $b \not=a$ such that ${\mu}_{l}(b) = \max \{{\mu}_{l}(i): i \in \{0,1, \dots, {{\alpha}_{l}}-1 \} \setminus \{a\}\}$. 
By our choice of $b$, we have that
\[{\mu}_{l}(b) \geq \frac{1 - {\eta}_{l}}{{\alpha}_{l} -1}.\]
This, together with 
\[\varepsilon < \frac{1}{16} \frac{1 - {\eta}_{l}}{{\alpha}_{l}},\]
implies 
\[\varepsilon < \frac{1}{16}
{\mu}_{l}(b) \le \frac{1}{16}
{\mu}_{l}(a).\] 
Applying Lemma~\ref{2ptsNotMixing} to $l$ and $a,b$ and $\varepsilon$, the conclusion follows.
\end{proof}
The next result follows from Lemma \ref{LemOverlap}.
 \begin{Lemma} \label{LemMixreverse} Suppose $\{\alpha_i\}$ and $\{\mu_i\}$ are such that  $T_f$ is continuous.  If 
  $T_{f}$ is topologically mixing, then  $\lim_{i} \frac{1 -{\eta}_{i}}{\alpha_i} = 0$.
   \end{Lemma} 

\begin{proof} To obtain a contradiction, suppose that $\limsup  \frac{1}{16}\frac{1 -{\eta}_{i}}{\alpha_i} > \varepsilon$ for some $\varepsilon >0$. 
As $T_{f}$ is topologically mixing,  by Theorem \ref{MixingBDP}, there exist $k_{0} \geq 1$ and measurable sets 
$\{B_{k}\}_{k=k_{0}}^{\infty}$  such that 
\[\mu({\mathcal A} \setminus B_{k}) < \varepsilon \text{ and } B_{k} \cap f^{k}(B_{k}) = \emptyset\   \text{ for all } k \geq k_{0}.\]
Choose $l$ large enough so that $ {\prod}_{i=0}^{l-1} {\alpha}_{i} >k_{0}$ and  $\frac{1}{16}\frac{1 -{\eta}_{l}}{\alpha_l} > \varepsilon$.  
Applying Lemma \ref{LemOverlap}, there exists  $m \ge  {\prod}_{i =0}^{l-1} {\alpha}_{i} > k_0$ such that, for all $B \subseteq {\mathcal A}$  with $\mu(B) > 1 - \varepsilon$, we have 
that $f^{m}(B) \cap B \not= \emptyset$. 
Hence, as $B_{m}$ has the property that $\mu(B_{m}) > 1 - \varepsilon$, we have that $f^{m}(B_{m}) \cap B_{m} \neq \emptyset$, yielding a contradiction. 
\end{proof}
{\bf Proof of Theorem \ref{MixingMainBounded}. } The
``if'' part follows from Lemma \ref{MixingForward} and the ``only if'' part 
 follows from Lemma \ref{LemMixreverse} as $\{\alpha_i\}$ bounded and $\lim_{i} \frac{1 -{\eta}_{i}}{\alpha_i} = 0$ imply that $\lim_{i}  {\eta}_{i} = 1$.
\qed

{\bf Proof of  Theorem~\ref{MixingExistence}}. Assume $\{\alpha_i\}$ is as in the hypothesis.  For each $n$, let $m_n = 2^{n+1}+1$.  We  define $\mu_n$ by  
\begin{align*}
    \mu_n(0)& = 1 - \frac{1}{m_n}   & \mu_n(i) & = \frac{1}{m_n (\alpha_n -1)}    \ \forall  1 \le i \le \alpha_n -1.
\end{align*} It is clear that $\mu_n$ is a probability measure with support $A_n$. 

We next show that  $\{\mu_i\}$ satisfies Condition ($\diamondsuit$). To this end, observe that for all $n \in \N$ and $0 \le i \le \alpha_n-1$ we have 
\begin{align*}
    \lambda_n(0) &= \left ( 1 - \frac{1}{m_n} \right)  m_n (\alpha_n -1)  &  \lambda_n(i) & \ge \frac{1}{m_n (\alpha_n -1)}.   
\end{align*}
Hence, for all $l \ge 1$ and $j \in A_{l}$, we have that 
\begin{align*}
    \lambda_{l}(j) {\prod}_{i=0}^{l-1} {\lambda}_{i}(0) & \ge \frac{1}{m_l (\alpha_l -1)} {\prod}_{i=0}^{l-1} \left ( 1 - \frac{1}{m_i} \right)  m_i (\alpha_i -1)\\
    & = \frac{1}{2^l} \cdot \frac{{\prod}_{i=0}^{l-1} (m_i-1)}{m_l} \cdot \frac{{\prod}_{i=0}^{l-1} \  2(\alpha_i-1)}{\alpha_l -1}\\
 & \ge \frac{1}{2^{l}} \cdot \frac{2^{l(l+1)/2}}{2^l+1} \cdot \frac{{\prod}_{i=0}^{l-1} \  \alpha_i}{\alpha_l}.
\end{align*}
We note that 
\[\lim_{l \rightarrow \infty} \frac{1}{2^{l}} \cdot\frac{2^{l(l+1)/2}}{2^l+1} = \infty, \]
and by hypothesis
\[ \liminf_l \frac{{\prod}_{i=0}^{l-1} \  \alpha_i}{\alpha_l} > 0,\]
verifying \[\inf _{l \ge 1,  j \in A_l} \lambda_{l}(j) {\prod}_{i=0}^{l-1} {\lambda}_{i}(0) > 0 \]  and Condition ($\diamondsuit$).

Now, by definition, we have that $\eta_n = 1 -2^{-(n+1)}$. Hence, by Lemma~\ref{MixingForward}, we have that $T_f$ is mixing. \\
\qed 

Before we construct Example~\ref{MixingExample}, we prove a lemma.
\begin{Lemma}\label{numbertheory}
Let $l \in \N$. Let $D_l = \{2^i-1: 0 \le i \le l\}$. Then, for each $j>0$, we have that $(D_l+j) \cap D_l$ has at most one element. 
\end{Lemma}
\begin{proof}
To obtain a contradiction, assume $(D_l+j) \cap D_l$ has at least two elements. Choose $u_1,v_1, u_2, v_2 \in D_l$ with $u_1 +j =v_1$, $u_2 + j = v_2$ and $v_1 < v_2$. Let $v_2 = 2^d-1$ and $u_2 = 2^c-1$. As $j >0$ we have that $d>c$.  Now, subtracting both sides, we have that $u_2-u_1 = v_2 -v_1$.  Note that $v_2 -v_1 \ge 2^d-1 - (2^{d-1} -1) = 2^{d-1} \ge 2^c$ where as $u_2 -u_1  \le 2^c -1$, yielding a contradiction. 
\end{proof}
{\bf Proof of Example \ref{MixingExample}} We let $\alpha_n = 2^{n+2}$,  $D_n =\{2^i-1: 0\le i \le n\}$ and $m_n =2^{n+2}$. We  define $\mu_n$ on $A_n$ by  
\begin{align*} 
    \mu_n(i)& = \left (1 - \frac{1}{m_n} \right ) \frac{1} {|D_n|} \ \ \forall i \in D_n  & \mu_n(i) & = \frac{1} {m_n(\alpha_n -|D_n|)}    \ \forall  i \in A_n \setminus D_n.
\end{align*}
We note that $\mu_n$ is a probability measure which is uniformly distributed on $D_n$ and $A_n \setminus D_n$. It is clear that $\lim_i \eta_i= 0$ as $\lim_i |D_i| = \lim_i (\alpha_i -|D_i|) =  \infty$. It remains to show that $\{\mu_i\}$ satisfies Condition ($\diamondsuit$) and that $T_f$ is mixing.

Let us  next show that  $\{\mu_i\}$ satisfies Condition ($\diamondsuit$). To this end, observe that for all $n \in \N$ and $0 \le i \le \alpha_n-1$ we have 
\begin{align*}
    \lambda_n(0) &= \left ( 1 - \frac{1}{m_n} \right)   \frac{m_n(\alpha_n -|D_n|)}{|D_n|}  &  \lambda_n(i) & \ge \frac{1}{m_n (\alpha_n -|D_n|)}.   
\end{align*}
As before,  we have that, for all $l \ge 1$ and $j \in A_{l}$,
\begin{align*}
    \lambda_{l}(j) {\prod}_{i=0}^{l-1} {\lambda}_{i}(0) & \ge \frac{1}{m_l (\alpha_l -|D_l|)} {\prod}_{i=0}^{l-1} \left ( 1 - \frac{1}{m_i} \right)   \frac{m_i(\alpha_i -|D_i|)}{|D_i|}\\
    & =  \frac{{\prod}_{i=0}^{l-1} (m_i-1)}{m_l} \cdot \frac{{\prod}_{i=0}^{l-1}    \frac{(\alpha_i-|D_i|)}{|D_i|}}{\alpha_l -|D_l|}\\
    & \ge \frac{3}{8} \cdot \frac{\prod_{i=0}^{l-1}3}{2^{l+2}}\\
    & \ge \frac{9}{64},
\end{align*}
verifying Condition ($\diamondsuit$).

We next apply Theorem~\ref{MixingBDP} to show that $T_f$ is mixing. To this end, let
$\varepsilon >0$. Let $l_0 \in \N$ be such that, for all $l \ge l_0$, we have 
\[ \left (1 - \frac{1}{m_l} \right) \left(1 - \frac{1}{m_{l+1}} \right ) \left (\frac{|D_l|-2}{|D_l|} \right)  \left (\frac{|D_{l+1}|-1}{|D_{l+1}|} \right )> 1- \varepsilon. \]
Let $k_0= \prod_{i=0}^{l_0} \alpha_i$. Note that, if $k \ge k_0$, then the representation of $k$ in terms of $\{\alpha_i\}$ has nonzero in some position $l \ge l_0$.

Now, for each $k \ge k_0$, we define $B_k$ so that  $\mu (B_k) \ge 1-\varepsilon$ and $f^k(B_k) \cap B_k = \emptyset$. Fix $k \ge k_0$. Let $l$ be the largest positive integer so that representation of $k$ in terms of $\{\alpha_i\}$ has nonzero in the $l^{th}$-position. Let $j \in A_l$ be the value of $k$ in the $l^{th}$-position. Applying Lemma \ref{numbertheory} to $D_l$, $j$ and $j+1$, we may choose $E_l \subseteq D_l$ so that $|E_l| \ge |D_l| -2$, $ (E_l +j) \cap E_l =\emptyset$ and $(E_l + (j+1) )\cap E_l=\emptyset$. (Here, $+$ is carried out in $\N$.) Let $F_{l+1} = D_{l+1}\setminus \{0\}$. Notice that $(F_{l+1} +1) \cap F_{l+1} = \emptyset$. 

We define
\[B_k= \left \{ \{x_i\} \in {\mathcal A}: x_l \in E_l \ \& 
\ x_{l+1} \in F_{l+1} \right \}.\]
As $l \ge l_0$, we have that 
\begin{align*}
  \mu(B_k) &= \mu_l (E_l) \cdot \mu_{l+1}(F_{l+1})\\
   &\ge \left (1 - \frac{1}{m_l} \right) \left ( \frac{|D_l|-2}{|D_l|} \right)  \left(1 - \frac{1}{m_{l+1}} \right ) \left ( \frac{|D_{l+1}|-1}{|D_{l+1}|} \right) \\
   &> 1- \varepsilon.
\end{align*}
 To complete the proof, we will show that $f^k(B_k) \cap B_k = \emptyset$ by considering two cases.
 
In the first case, assume that $2^l+j+1 \le 2^{l+2}-1$. Let $\{x_i\} \in B_k$. We have that $\{x_i\} \in C(k,B_k,l,0) \cup C(k,B_k,l,1)$. If $\{x_i\} \in C(k,B_k,l,0)$, then $f^k(\{x_i\}) \notin B_k$ as $(E_l +j) \cap E_l = \emptyset$. If $\{x_i\} \in C(k,B_k,l,1)$, then $f^k(\{x_i\}) \notin B_k$ as $(E_l +(j+1)) \cap E_l = \emptyset$. Hence, in this case we have that $f^k(B_k) \cap B_k = \emptyset$.
 
Now, assume that $2^l+j+1 \ge  2^{l+2}$. Note that $j > 2^l$.  Let $\{x_i\} \in B_k$. As $j > 2^l$, we have that $f^k(\{x_i\})$ has the property that its $l^{th}$ coordinate is not in $E_l$ or  its $l^{th}$ coordinate is in $E_l$ and we have a carry over in the $(l+1)^{st}$ position, implying that $(l+1)^{st}$ coordinate of  $f^k(\{x_i\})$ is in $F_{l+1}+1$. Summarizing, we have that
 \begin{align*}
f^{k}(B_{k}) &  \subseteq  \{x \in {\mathcal A}: x_{j} \notin E_j \} \cup \{x \in {\mathcal A}: x_{j+1}\in (F_{l+1}+1) \}.\\
\end{align*}
As $(F_{l+1} +1) \cap F_{l+1} = \emptyset$, we have that $f^{k}(B_{k}) \cap B_{k} = \emptyset$ and completing the proof.
\qed
\section{Proof of Transitivity Results}
	
{\bf Proof of  Theorem \ref{TRANSCHAR}} 
$(\Leftarrow)$ Let $\varepsilon >0$. Choose integers $i \leq j$ so that ${\psi}_{i,j} > 1- \varepsilon$.  Let  $A \subseteq {\prod}_{s=i}^j A_{s}$ and   $h=(h_i,h_{i+1},...,h_j) \in  {\prod}_{s=i}^{j} {A}_{s}$  be such that  $\psi_{i,j}=\mu_{i,j}(A)$ and $(A+h) \cap A = \emptyset$.  Let $k$ be the positive integer whose representation has zero everywhere except in the $i, \ i+1,.., j$-th positions where there are respectively  $h_i,h_{i+1},...,h_j$. Now take $B = \prod_{s=0}^{i-1}{A}_{s} \times A \times  \prod_{s=j+1}^{\infty}{A}_{s}$.
Then, $\mu(B)   =   {\mu}_{i,j}(A) > 1- \varepsilon$. Moreover,  $f^{k}(B)  \cap  B = \emptyset$. Therefore, by Theorem \ref{TransBDP}, $T_{f}$ is transitive.  

$(\Rightarrow)$ Suppose $T_{f}$ is transitive. Fix $\varepsilon >0$. Then there exist  $k\ge 1$ and measurable set $B\subseteq X$ with 
$\mu(X{\setminus} B)< \varepsilon$ and $B\cap f^k(B)=\emptyset$. As $\mu$ is tight, we can assume that $B$ is compact. Now we consider projections of $B$. Recall that $p_{i,j}(B)$ is the projection of $B$ onto $\{i,\ldots ,j \}$ coordinates. For each $i$, let $B_{i} :=  p_{0,i}(B)$. 

We note that $\mu_{0,i}(B_i) \ge \mu(B) > 1-\varepsilon.$ To complete the proof, we show that for some $i$, $(k + B_{i}) \cap B_{i} = \emptyset$, implying that  $\psi_{0,i} > 1 - \varepsilon$. 

Let $l$ be such that the representation of $k$ in $\{\alpha_i\}$ has zero in the $l^{th}$ position and beyond. To obtain a contradiction, assume that, for all $i >l$, there exists $x^{(i)} \in B_{i}$ such that $k + x^{(i)} \in B_{i}$. For each $i$, let $\tilde{{x}}^{(i)} \in B$ be an extension of $x^{(i)}$, that is $p_{0,i}(\tilde{{x}}^{(i)})= x^{(i)}$. Notice that $k + x^{(i)} = p_{0,i}( k + \tilde{{x}}^{(i)})$ for all $i >k$. 
As $B$ is compact, we may choose a subsequence $\{\tilde{x}^{(n_{i})}\}$ of $\{\tilde{x}^{(i)}\}$ which converges to some $x$ in $B$. Moreover, we may require that $p_{0,n_{i}}(x)= p_{0,n_{i}}(\tilde{x}^{(n_{i})}) $ for all $i$. Therefore, for all $i \ge k$
\[ p_{0,n_i} (k +x )  =  k +p_{0,n_{i}}(x)= k+ p_{0,n_{i}}(\tilde{x}^{(n_{i})})  =  k + x^{(n_i)} \in B_{n_i}.\]

 Hence, we have shown that for arbitrary large $i$, we have that
first $n_{i}$ coordinates of $k+x = f^{k}(x)$ equals first $n_{i}$ coordinates of some element of $B$. As $B$ is closed, we have that $f^{k}(x) \in B$. However, this contradicts that $f^{k}(B) \cap B = \emptyset$.
Hence, we have that for some $i$, $(k+B_i) \cap B_i = \emptyset$.\\
\qed

{\bf Proof of Corollary \ref{CorTRANSSUF}}
The corollary follows from the fact that $\eta_i \leq \psi_i$ and Theorem \ref{TRANSCHAR}.\\
\qed 

{\bf Proof of Theorem~\ref{TransExistence1}}
Let $\{\alpha_n\}$ be such that $\liminf_n \alpha_n = t < \infty$. We may choose a strictly increasing sequence of positive integers $\{n_k\}$  such that $\alpha_{n_k} = t$ for all $k$. Moreover, we may also require that for infinitely many $n$'s not belonging to $\{n_k\}$ we also have that $\alpha_n =t$. We define $\mu_n$ in the following fashion: if $n$ does not belong to $\{n_k\}$, then $\mu_n$ has the uniform distribution on $A_n$. For $n =n_k$, we let $\mu_n$ be such that 
\begin{align*}
    \mu_n(0)& = 1 - \frac{1}{m_k}   & \mu_n(i) & = \frac{1}{m_k (\alpha_n -1)}    \ \forall  1 \le i \le \alpha_n -1,
\end{align*}
where $m_k = k+3$.
We note that $\lim_n \frac{1-\eta_n}{\alpha_n} \neq 0$ as there are infinitely many $n$'s for which $\eta_n = 1/t$ and $\alpha_n =t$. We also have that $ \limsup_n \eta _n = 1$. In light of Lemma~\ref{LemMixreverse} and Corollary~\ref{CorTRANSSUF}, $T_f$ is transitive but not mixing, provided that we show that $T_f$ is continuous, or equivalently, that $\{\mu_n\}$ satisfies Condition ($\diamondsuit$). To this end, fix $l \ge 1$ and $j \in A_{l}$. We note that if $l$ does not belong to $\{n_k\}$, then 
\begin{align*}
    \lambda_{l}(j) {\prod}_{i=0}^{l-1} {\lambda}_{i}(0) & =  {\prod}_{k=0}^{p} (m_k-1) (\alpha_{n_k} -1) \ge 1,
\end{align*}
where $p$ is the largest integer so that $n_p \le l-1$.  In the case that $l=n_p$ for some $p$, we have that 
\begin{align*}
    \lambda_{l}(j) {\prod}_{i=0}^{l-1} {\lambda}_{i}(0) & \ge \frac{1}{m_p(\alpha_{n_p}-1)}  {\prod}_{k=0}^{p-1} (m_k-1) (\alpha_{n_k} -1)\\
    & \ge \frac{m_{p-1}-1}{m_p} \cdot \frac{1}{t-1}\\
    &\ge \frac{2}{4}\cdot \frac{1}{t-1}.
\end{align*}
Hence, 
\begin{align*}
   \liminf_{l \ge 1, j \in A_l}  \lambda_{l}(j) \prod_{i=0}^{l-1} {\lambda}_{i}(0) & \ge  \frac{1}{2}\cdot \frac{1}{t-1},
\end{align*}
verifying that $\{\mu\}$ satisfies Condition ($\diamondsuit$) in this case. \\
\qed

{\bf Proof of Theorem~\ref{TransExistence2}}. 
Assume the hypotheses. To avoid trivial cases, we assume that $\alpha_n \ge 4$ for all $n$. (If $\alpha_n < 4$ for infinitely many $n$'s, then we are done by Theorem~\ref{TransExistence1}.)
We define $\mu_n$ on $A_n$ in the following fashion,
\begin{align*}
    \mu_n(i)& = \frac{1}{2} \left (1 - \frac{1}{m_n} \right)  \ \  \forall i \in \{0,1\}   & \mu_n(i) & = \frac{1}{m_n (\alpha_n -2)}    \  \ \forall  i \notin \{0,1\},
\end{align*}
where $m_n = 2^n+1$. 
 Let us show that $\{\mu_n\}$ satisfies Condition ($\diamondsuit$). To this end, observe that for all $n \in \N$ and $1 \le i \le \alpha_n-1$ we have 
\begin{align*}
    \lambda_n(0) &= \frac{1}{2}\left ( 1 - \frac{1}{m_n} \right)  m_n (\alpha_n -2)  &  \lambda_n(i) & \ge \frac{1}{m_n (\alpha_n -2)}.   
\end{align*}
Hence, for all $l \ge 1$ and $j \in A_{l}$, we have that 
\begin{align*}
    \lambda_{l}(j) {\prod}_{i=0}^{l-1} {\lambda}_{i}(0) & \ge \frac{1}{m_l (\alpha_l -2)} {\prod}_{i=0}^{l-1} \  \frac{1}{2}\left ( 1 - \frac{1}{m_i} \right)  m_i (\alpha_i -2)\\
    & = \frac{1}{2^{2l}} \cdot \frac{{\prod}_{i=0}^{l-1} (m_i-1)}{m_l} \cdot \frac{{\prod}_{i=0}^{l-1} \  2(\alpha_i-2)}{\alpha_l -2}\\
    & \ge \frac{1}{2^{2l}} \cdot \frac{2^{l(l-1)/2}}{2^l+1} \cdot \frac{{\prod}_{i=0}^{l-1} \  \alpha_i}{\alpha_l}.
\end{align*}
We note that 
\[\lim_{l \rightarrow \infty} \frac{1}{2^{2l}} \cdot\frac{2^{l(l-1)/2}}{2^l+1} = \infty, \]
and, by hypothesis,
\[ \liminf_l \frac{{\prod}_{i=0}^{l-1} \  \alpha_i}{\alpha_l} > 0,\]
verifying \[\inf _{l \ge 1,  j \in A_l} \lambda_{l}(j) {\prod}_{i=0}^{l-1} {\lambda}_{i}(0) > 0 \]  and Condition ($\diamondsuit$).

We next show that $T_f$ is transitive. Let $D_n \subseteq A_n$ be defined as $\{0,1\}$. As $\alpha_n \ge 4$ for all $n$, we have that $(D_n+2)\cap D_n = \emptyset$. Hence, \[\psi_n \ge \mu_n(D_n) = 1-\frac{1}{m_n} = 1-\frac{1}{2^n+1}.\]
Now, by Theorem~\ref{TRANSCHAR}, it follows that $T_f$ is transitive.

Let us finally show that $T_f$ is not mixing. Let $0 < \varepsilon < \frac{1}{64}$. Note that, for all $l \in \N$, we have that $\mu_l (0) = \mu_l (1) \ge \frac{1}{4}$. Therefore, for all $l$ we have that 
\[ \varepsilon < \frac{1}{16} \min \{\mu_l(0), \mu_l (1)\}.\]
Hence, by Lemma~\ref{2ptsNotMixing}, we have that for every $l$, there exists $m > l$ such that for all $B$ with $\mu(B) > 1 - \varepsilon$ we have that $f^m(B) \cap B \neq \emptyset$. By Theorem~\ref{MixingBDP}, we have that $T_f$ cannot be mixing.
\qed

The following lemma will aid us in the proof of  Theorem~\ref{TRANSNEC}
\begin{Lemma} \label{PROPT} Suppose $\{\alpha_i\}$ and $\{\mu_i\}$ are such that  $T_f$ is continuous. 
If $T_f$  is transitive, then for every $\varepsilon >0$,  there is an open set $U \subseteq {\mathcal A}$ with $\mu (U) < \varepsilon$ and a positive integer $k$ such that $\mu(f^k(U)) > 1-\varepsilon$.
\end{Lemma}
\begin{proof} 
Let $\varepsilon >0$. As $T_f$ is transitive, by Theorem~\ref{TransBDP}, we may choose 
$B\subseteq {\mathcal A}$ with $\mu(B) >  1- \varepsilon$ and $B\cap f^k(B)=\emptyset$. 
By the fact that every Borel probability measure is tight, we can choose a closed set $C \subseteq B$ so that $\mu (C) > 1-\varepsilon$. Hence, after renaming, we may assume that $B$ is closed. Let $U = {\mathcal A} \setminus B$. Then, $\mu(U) <\varepsilon$. Moreover, as $f$ is a bijection, since $f=f_{\alpha}$, and $B\cap f^k(B)=\emptyset$, $f^k(U)$ contains $B$ and hence $\mu (f^k(U)) > 1 -\varepsilon$. 
\end{proof}
 {\bf Proof of Theorem~ \ref{TRANSNEC}}  Let $T_{f}$ be transitive and assume, by contradiction, that the sequence $\{\rho_{n}\}_{n=0}^{\infty}$ is bounded. 
Let $M>2$ satisfy 
\[\sup_{n \in \N} \rho_{n} \leq M.\]
Let $D := [{d}_0,{d}_1,\ldots {d}_j]$ be a fixed basic cylinder, then $\mu(f^k(D))\leq M\mu (D)$ for all $k \geq 1$. Indeed, as $f^k(D)$ is a cylinder of the form $[{e}_{0},{e}_{1},\ldots, {e}_j]$, we have \begin{align*}
\mu(f^k(D)) & =   {\mu}_{0}({e}_{0}) {\mu}_{1}({e}_{1})\cdots {\mu}_{j}({e}_{j})\\
&  \leq  {\eta}_{0} {\eta}_{1} \cdots {\eta}_{j} \\
& \leq  M({\delta}_{0} {\delta}_{1} \cdots {\delta}_{j}) \\
& \leq   M({\mu}_{0}(d_0) {\mu}_{1}(d_1) \cdots {\mu}_{j}(d_j))\\
& =  M \mu(D)\\
\end{align*}
As each open set can be written as a countable union of pairwise disjoint basic cylinders, by the additivity property of measures, we have that, for all open sets $U$ and all positive integers $k$, 
\[\mu(f^k(U))  \le M \mu(U).\]
Now, let $\varepsilon: = 1/M^2$, and U be any open set with $\mu (U) <\varepsilon$. Then, for all $k \geq 1$,  \[\mu(f^k(U)) < M \varepsilon = 1/M < 1 - \varepsilon.\]
 By Lemma \ref{PROPT} $T_{f}$ cannot be transitive, yielding a contradiction. 
\qed 
\bibliographystyle{abbrv}
\bibliography{Chaos}

\begin{thebibliography}{10}

\bibitem{BDP}
F.~Bayart, U.~B. Darji, and B.~Pires.
\newblock Topological transitivity and mixing of composition operators.
\newblock {\em J. Math. Anal. Appl.}, 465(1):125--139, 2018.

\bibitem{BM}
F.~Bayart and E.~Matheron.
\newblock {\em Dynamics of linear operators}, volume 179 of {\em Cambridge
  Tracts in Mathematics}.
\newblock Cambridge University Press, Cambridge, 2009.

\bibitem{BM2}
F.~Bayart and E.~Matheron.
\newblock ({N}on-)weakly mixing operators and hypercyclicity sets.
\newblock {\em Ann. Inst. Fourier (Grenoble)}, 59(1):1--35, 2009.

\bibitem{BR}
F.~Bayart and I.~Z. Ruzsa.
\newblock Difference sets and frequently hypercyclic weighted shifts.
\newblock {\em Ergodic Theory Dynam. Systems}, 35(3):691--709, 2015.

\bibitem{BerMess}
J.~{Bernardes}, Nilson~C. and A.~{Messaoudi}.
\newblock {Shadowing and structural stability in linear dynamical systems}.
\newblock {\em arXiv e-prints}, page arXiv:1902.04386, Feb 2019.

\bibitem{BDP2}
N.~C. Bernardes, U.~B. Darji, and B.~Pires.
\newblock Li-yorke chaos for composition operators on $l^p$-spaces, 2018.

\bibitem{BBM}
N.~C. Bernardes, Jr., A.~Bonilla, V.~M\"{u}ller, and A.~Peris.
\newblock Li-{Y}orke chaos in linear dynamics.
\newblock {\em Ergodic Theory Dynam. Systems}, 35(6):1723--1745, 2015.

\bibitem{BCDMP}
N.~C. Bernardes, Jr., P.~R. Cirilo, U.~B. Darji, A.~Messaoudi, and E.~R.
  Pujals.
\newblock Expansivity and shadowing in linear dynamics.
\newblock {\em J. Math. Anal. Appl.}, 461(1):796--816, 2018.

\bibitem{BK}
L.~Block and J.~Keesling.
\newblock A characterization of adding machine maps.
\newblock {\em Topology Appl.}, 140(2-3):151--161, 2004.

\bibitem{BK2}
L.~S. Block and W.~A. Coppel.
\newblock {\em Dynamics in one dimension}, volume 1513 of {\em Lecture Notes in
  Mathematics}.
\newblock Springer-Verlag, Berlin, 1992.

\bibitem{BDD}
D.~Bongiorno, U.~B. Darji, and L.~Di~Piazza.
\newblock Rolewicz-type chaotic operators.
\newblock {\em J. Math. Anal. Appl.}, 431(1):518--528, 2015.

\bibitem{CS}
G.~Costakis and M.~Sambarino.
\newblock Topologically mixing hypercyclic operators.
\newblock {\em Proc. Amer. Math. Soc.}, 132(2):385--389, 2004.

\bibitem{DDS}
E.~D'Aniello, U.~B. Darji, and T.~H. Steele.
\newblock Ubiquity of odometers in topological dynamical systems.
\newblock {\em Topology Appl.}, 156(2):240--245, 2008.

\bibitem{DHS}
E.~D'Aniello, P.~D. Humke, and T.~H. Steele.
\newblock The space of adding machines generated by continuous self maps of
  manifolds.
\newblock {\em Topology Appl.}, 157(5):954--960, 2010.

\bibitem{DS}
E.~D'Aniello and T.~H. Steele.
\newblock Prevalance and structure of adding machines for cellular automata.
\newblock {\em J. Math. Anal. Appl.}, 352(2):856--860, 2009.

\bibitem{DASI}
A.~I. Danilenko and C.~E. Silva.
\newblock Ergodic theory: non-singular transformations.
\newblock In {\em Mathematics of complexity and dynamical systems. {V}ols.
  1--3}, pages 329--356. Springer, New York, 2012.

\bibitem{DH}
A.~H. Dooley and T.~Hamachi.
\newblock Nonsingular dynamical systems, {B}ratteli diagrams and {M}arkov
  odometers.
\newblock {\em Israel J. Math.}, 138:93--123, 2003.

\bibitem{gm}
S.~Grivaux and E.~Matheron.
\newblock Invariant measures for frequently hypercyclic operators.
\newblock {\em Adv. Math.}, 265:371--427, 2014.

\bibitem{G0}
K.-G. Grosse-Erdmann.
\newblock Hypercyclic and chaotic weighted shifts.
\newblock {\em Studia Math.}, 139(1):47--68, 2000.

\bibitem{GP}
K.-G. Grosse-Erdmann and A.~Peris~Manguillot.
\newblock {\em Linear chaos}.
\newblock Universitext. Springer, London, 2011.

\bibitem{HOC}
M.~Hochman.
\newblock Genericity in topological dynamics.
\newblock {\em Ergodic Theory Dynam. Systems}, 28(1):125--165, 2008.

\bibitem{MP}
F.~Mart\'{\i}nez-Gim\'{e}nez and A.~Peris.
\newblock Chaos for backward shift operators.
\newblock {\em Internat. J. Bifur. Chaos Appl. Sci. Engrg.}, 12(8):1703--1715,
  2002.

\bibitem{ORN}
D.~S. Ornstein.
\newblock On invariant measures.
\newblock {\em Bull. Amer. Math. Soc.}, 66:297--300, 1960.

\bibitem{RO}
S.~Rolewicz.
\newblock On orbits of elements.
\newblock {\em Studia Math.}, 32:17--22, 1969.

\bibitem{S1}
H.~N. Salas.
\newblock Hypercyclic weighted shifts.
\newblock {\em Trans. Amer. Math. Soc.}, 347(3):993--1004, 1995.

\bibitem{S2}
H.~N. Salas.
\newblock Supercyclicity and weighted shifts.
\newblock {\em Studia Math.}, 135(1):55--74, 1999.

\bibitem{SM}
R.~K. Singh and J.~S. Manhas.
\newblock {\em Composition operators on function spaces}, volume 179 of {\em
  North-Holland Mathematics Studies}.
\newblock North-Holland Publishing Co., Amsterdam, 1993.

\end{thebibliography}
\end{document}